\documentclass[12pt]{amsart}
\usepackage{amsfonts,graphicx,amssymb,amscd,amsmath,enumerate,verbatim,calc}

\def\NZQ{\mathbb}               

\def\ZZ{{\NZQ Z}}
\def\RR{{\NZQ R}}

\def\eb{{\bold e}}

\def\opn#1#2{\def#1{\operatorname{#2}}} 
\opn\pd{pd} 
\opn\rk{rk}
\opn\rank{rank}
\opn\depth{depth} 
\opn\grade{grade} 
\opn\height{height}
\opn\embdim{emb\,dim} 
\opn\codim{codim}
\opn\Tr{Tr} 
\opn\bigrank{big\,rank}
\opn\lcm{lcm}
\opn\reg{reg} 
\opn\ini{in} 
\opn\size{size}
\opn\mult{mult}
\opn\dist{dist}
\opn\cone{cone}
\opn\lex{lex}
\opn\rev{rev}
\opn\div{div}
\opn\Div{Div}
\opn\cl{cl}
\opn\Cl{Cl}
\opn\Syz{Syz} \opn\Im{Im} \opn\Ker{Ker} \opn\Coker{Coker}
\opn\Hom{Hom} \opn\Tor{Tor} \opn\Ext{Ext}
\opn\End{End} \opn\Aut{Aut} \opn\id{id} \opn\nat{nat}
\opn\mod{mod} \opn\ord{ord}
\opn\aff{aff} \opn\con{conv} \opn\relint{relint} \opn\st{st}
\opn\lk{lk} \opn\cn{cn} \opn\core{core} \opn\vol{vol}
\opn\link{link} \opn\star{star}
\def\Ac{{\mathcal A}}

\def\Hc{{\mathcal H}}

\def\Gc{{\mathcal G}}
\def\Fc{{\mathcal F}}
\def\Pc{{\mathcal P}}
\def\Qc{{\mathcal Q}}

\newtheorem{Theorem}{Theorem}[section]
\newtheorem{Lemma}[Theorem]{Lemma}
\newtheorem{Corollary}[Theorem]{Corollary}
\newtheorem{Proposition}[Theorem]{Proposition}
\newtheorem{Remark}[Theorem]{Remark}

\begin{document}
\title{Non-normal very ample polytopes and their holes}
\author{Akihiro Higashitani}
\thanks{
\;\;\;\; The author is supported by JSPS Research Fellowship for Young Scientists. }
\address{Akihiro Higashitani,
Department of Pure and Applied Mathematics,
Graduate School of Information Science and Technology,
Osaka University,
Toyonaka, Osaka 560-0043, Japan}
\email{a-higashitani@cr.math.sci.osaka-u.ac.jp}

\begin{abstract}
In this paper, 
we show that for given integers $h$ and $d$ with $h \geq 1$ and $d \geq 3$, 
there exists a non-normal very ample integral convex polytope of dimension $d$ 
which has exactly $h$ holes. 
\end{abstract}

\maketitle

\section*{Introduction}

The normality and the very ampleness of integral convex polytopes are 
of importance in the several points of view, 
e.g., not only combinatorics on convex polytopes 
but also toric geometry and commutative algebra. 
In particular, normal or very ample integral convex polytopes appearing 
in the context of toric geometry are well studied (cf. \cite{BHHLNPSS, Bruns, Ogata02, Ogata03, Ogata11}). 
To determine whether a given integral convex polytope is normal (very ample) or not 
is a fundamental but fascinating problem. (See \cite{HHKO, OH98, OH10}.) 
In this paper, we will show the existence of 
non-normal very ample integral convex polytopes for general dimensions.

Let $\Pc \subset \RR^{N}$ be an integral convex polytope, which is a convex polytope 
all of whose vertices are contained in $\ZZ^N$, of dimension $d$. 
Define $\widetilde{\Pc} \subset \RR^{N+1}$ to be the convex hull of the points 
$(\alpha,1) \in \RR^{N+1}$ with $\alpha \in \Pc$ and let $\Ac_\Pc = \widetilde{\Pc} \cap \ZZ^{N+1}$. 
We say that $\Pc$ is {\em normal} if $\Pc$ satisfies 
$$ \RR_{\geq 0}\Ac_\Pc \cap \ZZ \Ac_\Pc=\ZZ_{\geq 0}\Ac_\Pc.$$ 
Moreover, we say that $\Pc$ is {\em very ample} if the set 
$$ (\RR_{\geq 0}\Ac_\Pc \cap \ZZ \Ac_\Pc) \setminus \ZZ_{\geq 0}\Ac_\Pc $$ is finite 
and we call the elements of 
$(\RR_{\geq 0}\Ac_\Pc \cap \ZZ \Ac_\Pc) \setminus \ZZ_{\geq 0}\Ac_\Pc$ the {\em holes} of $\Pc$. 
In particular, when $\Pc$ is normal, $\Pc$ is also very ample.

In addition, for a positive integer $k$, 
we say that $\Pc$ is {\em $k$-normal} if for each $n=k,k+1,\ldots$ 
and for each $\alpha \in n \Pc \cap \ZZ^N$, where $n \Pc=\{ \, n \alpha \, : \, \alpha \in \Pc \, \}$, 
there exist $n$ integer points $\alpha_1, \ldots, \alpha_n$ 
belonging to $\Pc \cap \ZZ^N$ 
such that $\alpha = \alpha_1 + \cdots + \alpha_n$. 

Let us assume 
$$
N=d \;\;\text{and} \;\; \ZZ \Ac_\Pc = \ZZ^{d+1}. 
$$
Then $\Pc$ is normal if and only if $\Pc$ is 1-normal, 
which is also called that $\Pc$ has the {\em integer decomposition property}. 
Moreover, $\Pc$ is very ample if and only if 
$\Pc$ is $k$-normal for some sufficiently large positive integer $k$. 
A definition of very ampleness described in, e.g., \cite{BruGub09, CLS, Ogata11} 
is also equivalent to ours. See \cite[Exercise 2.23]{BruGub09}.

It often happens that for some class of integral convex polytopes, 
its normality is equivalent to its very ampleness. 
In other words, a very ample integral convex polytope is always normal among some class of polytopes. 
For example, edge polytope is a typical example (cf \cite{OH98}). 
Thus, the following is a quite natural question: 

\smallskip

\begin{center}
Does there exist an integral convex polytope which is not normal but very ample ?
\end{center}

\smallskip

In \cite[Example 5.5.1]{BruGub02}, Bruns and Gubeladze succeed to giving 
the first example of a non-normal very ample integral convex polytope, 
which is of dimension 5 and can be obtained from a triangulation of a real projective plane. 
Recently, they provide the second example in \cite[Exercise 2.24]{BruGub09}, 
which is of dimension 3. 
Moreover, in \cite[Section 2]{Ogata11}, Ogata generalizes the second example 
and establishes infinitely many non-normal very ample integral convex polytopes of dimension 3. 

On the other hand, for general dimensions, 
no example of non-normal very ample integral convex polytopes is known. 
In \cite[Section 1]{Ogata11}, such polytopes have been proposed. But, unfortunately, 
it turns out that \cite[Proposition 1]{Ogata11} does not hold in general. 
In fact, for example, let $n=4$ and $q=2$. Then \cite[Proposition 1]{Ogata11} says that 
\begin{eqnarray*}
&&P_2=\con(\{(0,0,0,0),(1,0,0,0),(0,1,0,0),(0,0,1,0),(1,1,1,2)\} \cup \\
&&\quad\quad\quad\quad\quad 
\{(0,0,0,1),(1,0,0,1),(0,1,0,1),(0,0,1,1),(1,1,1,3)\}) \subset \RR^4
\end{eqnarray*}
is very ample, while it is not normal by \cite[Proposition 2]{Ogata11}. 
Remark that $\ZZ \Ac_{P_2}=\ZZ^5$. 
However, for {\em every} positive integer $m$, since we have 
$$(m,1,1,1,m+1)=\frac{2m-1}{2}(1,0,0,0,1)+\frac{1}{2}(0,1,0,0,1)+
\frac{1}{2}(0,0,1,0,1)+\frac{1}{2}(1,1,1,2,1),$$ 
one has $(m,1,1,1,m+1) \in \RR_{\geq 0} \Ac_{P_2} \cap \ZZ^5$, 
whereas one can see that $(m,1,1,1,m+1) \not\in \ZZ_{\geq 0} \Ac_{P_2}$. 
This shows that there exist infinitely many holes. 
That is to say, this is NOT very ample. 




In this paer, we present a non-normal very ample integral convex polytope 
for general dimensions having an additional property. 
The following is our main theorem of this paper. 
\begin{Theorem}\label{main}
Let $h$ and $d$ be integers with $h \geq 1$ and $d \geq 3$. 
Then there exists a non-normal very ample integral convex polytope of dimension $d$ 
which has exactly $h$ holes. 
\end{Theorem}

Let $h$ and $d$ be as above and let 
\begin{align*}
&u_i=
\begin{cases}
{\bf 0}, \;\;\;\;\;\;\;\; &i=1, \\
\eb_d, &i=2, \\
\eb_2+\cdots+\eb_{d-1}, &i=3, \\
h(\eb_2+\cdots+\eb_{d-1}+\eb_d), &i=4, \\
(h-1)(\eb_2+\cdots+\eb_{d-1})+h\eb_d, &i=5, \\
h(\eb_2+\cdots+\eb_{d-1})+(h-1)\eb_d, &i=6, \\
\eb_1+4\eb_d, &i=7, \\
\eb_1+5\eb_d, &i=8, \\
\eb_1+\eb_2+\cdots+\eb_{d-1}, &i=9, \\
\eb_1+\eb_2+\cdots+\eb_{d-1}+\eb_d, &i=10, 
\end{cases}\\
&v_i=\eb_i,  \;\;\;\;\;\quad\;\;\;\; i=2,\ldots,d-1, \\
&v_i'=\eb_i+\eb_d, \;\;\;\;\; i=2,\ldots,d-1, 
\end{align*}
where ${\bf 0}=(0,\ldots,0) \in \RR^d$ and $\eb_1,\ldots,\eb_d$ are the unit coordinate vectors of $\RR^d$. 
We define the integral convex polytope $\Pc_{h,d} \subset \RR^d$ by setting the convex hull of 
$$\{u_1, \ldots, u_{10}\} \cup \{v_i, v_i' : i=2,\ldots,d-1 \}.$$ 

In this paer, 
we will show that $\Pc_{h,d}$ enjoys the required properties, 
i.e., this is a non-normal very ample integral convex polytope 
of dimension $d$ which has exactly $h$ holes. 
It is immediate that $\ZZ (\Pc_{h,d} \cap \ZZ^d)=\ZZ^d$. 
Thus $\dim (\Pc_{h,d}) =d$ and $\ZZ \Ac_{\Pc_{h,d}} = \ZZ^{d+1}$. 

The strategy of our proof is as follows. 
After calculating all the facets of $\Pc_{h,d}$, 
we will first prove that $\Pc_{h,d} \cap \{(x_1,\ldots,x_d) \in \RR^d : x_1=0 \}$ 
is normal by using theory of Gr\"obner basis and 
the remaining facets of $\Pc_{h,d}$ are also normal in Section \ref{normality}. 
In Section \ref{Hilb}, we will analyze $2\Pc_{h,d} \cap \{(x_1,\ldots,x_d) \in \RR^d : x_1=1 \}$ 
and find $h$ holes of $\Pc_{h,d}$, which implies that $\Pc_{h,d}$ is not normal. 
At last, Section \ref{nohole} is devoted to showing that 
there is no hole except for such $h$ holes, also forcing $\Pc_{h,d}$ is very ample. 

\begin{Remark}{\em 
When $d=3$, there already exists a non-normal very ample integral convex polytope having exactly $h$ holes. 
In \cite[Example 15]{BHHLNPSS}, a non-normal very ample integral convex polytope 
$$Q_k=\con(\{{\bf 0}, \eb_1,\eb_2,\eb_3,\eb_1+\eb_3,\eb_2+\eb_3, \eb_1+\eb_2+k\eb_3,\eb_1+\eb_2+(k+1)\eb_3\}) \subset \RR^3$$
of dimension 3 is described. For a given positive integer $h$, let $k=h+3$. 
Then this polytope has $h$ holes. 

Now, remark that higher dimensional polytopes obtained by taking pyramid over $Q_k$ 
do not preserve very ampleness. 
In fact, for a positive integer $h$, let us consider an integral convex polytope 
$Q'=\con(\{(\alpha, 1) \in \RR^4: \alpha \in Q_{h+3}\} \cup \{{\bf 0}\})$ of dimension 4. 
Since $(1,1,3,2)$ is a hole of $Q_{h+3}$, $(1,1,3,2,2)$ is also a hole of $Q'$. 
Moreover, for every integer $m$ with $m \geq 2$, $(1,1,3,2,m)=(1,1,3,2,2)+(m-2)(0,0,0,0,1)$ 
is also a hole of $Q'$. This says the non-very ampleness of $Q'$. 
}\end{Remark}


\section{Normality of facets of $\Pc_{h,d}$}\label{normality}
In this section, we verify that the facets of $\Pc_{h,d}$ are all normal, 
which we shall use in Section \ref{nohole}.

For a hyperplane $\Hc \subset \RR^d$ defined by the equality 
$a_1x_1+\cdots+a_dx_d = b$, we write $\Hc^{(+)}$ (resp. $\Hc^{(-)}$)
for the closed half space defined by the inequality $a_1x_1+\cdots+a_dx_d \leq b$ 
(resp. $a_1x_1+\cdots+a_dx_d \geq b$). 
We define ten types of hyperplanes as follows: 
\begin{itemize}
\item[$\Hc_0$ :] $x_1 = 0$, 
\item[$\Hc_1$ :] $x_d = 0$, 
\item[$\Hc_{2,i}$ :] $-x_i=0$, 
\item[$\Hc_{3,i}$ :] $ - (d-4)x_i + \sum_{j\not=i, 2 \leq j \leq d-1}x_j-x_d=1$, 
\item[$\Hc_{4,i}$ :] $4x_1 - 4x_i-x_d = 0$, 
\item[$\Hc_{5,i}$ :] $-4x_1 - x_i + x_d = 1$, 
\item[$\Hc_{6,i}$ :] $x_1 - (d-3)x_i + \sum_{j\not=i, 2 \leq j \leq d-1}x_j=1$, 
\item[$\Hc_{7,i}$ :] $(5h-5)x_1-((d-3)(5h-1)-4)x_i+(5h-1)\sum_{j\not=i, 2 \leq j \leq d-1}x_j + x_d = 5h$, 
\item[$\Hc_{8,i}$ :] $(h-5)x_1-(d-3)(h-1)x_i+(h-1)\sum_{j\not=i, 2 \leq j \leq d-1}x_j + x_d =h$, 
\item[$\Hc_{9,i}$ :] $(h-1)x_1-((d-3)h-1)x_i+h\sum_{j\not=i, 2 \leq j \leq d-1}x_j = h$, 
\end{itemize}
%
%
where $i=2,\ldots,d-1$. Then each hyperplane above is a supporting hyperplane of $\Pc_{h,d}$. 
Moreover, some routine works unable us to show that 
there is no facet except for the facets defined by the above $(8(d-2)+2)$ supporting hyperplanes. Hence, 
\begin{eqnarray}\label{facets}
\Pc_{h,d} = \Hc_0^{(-)} \cap \Hc_1^{(-)} \cap 
\left(\bigcap_{\substack{2 \leq j \leq 9, \\ 2 \leq i \leq d-1}} \Hc_{j,i}^{(+)} \right). 
\end{eqnarray}

Let $\Fc_0,\Fc_1,\Fc_{j,i}$, where $j=2,\ldots,9$ and $i=2,\ldots,d-1$, 
be facets of $\Pc_{h,d}$ defined by the corresponding hyperplanes 
$\Hc_0,\Hc_1$ or $\Hc_{j,i}$.


\bigskip

We prove the normality of $\Fc_0$. 
We employ some techniques using Gr\"obner basis. 
We refer the readers to \cite{HerzogHibi} for fudamental materials on Gr\"obner basis. 

\smallskip

Let 
\begin{align*}
&u_{3,j}=\frac{(h-1-j)u_3+ju_6}{h-1}=(j+1)(\eb_2+\cdots+\eb_{d-1})+j\eb_d, \;\; j=0,1,\ldots,h-1, \\
&u_{2,j}=\frac{(h-1-j)u_2+ju_5}{h-1}=j(\eb_2+\cdots+\eb_{d-1})+(j+1)\eb_d, \;\; j=0,1,\ldots,h-1, \\
&u_{1,j}=\frac{(h-j)u_1+ju_4}{h}=j(\eb_2+\cdots+\eb_{d-1})+j\eb_d, 
\;\;\;\;\;\;\;\quad\quad\quad j=0,1,\ldots,h. 
\end{align*}
Then $u_{3,1},\ldots,u_{3,h-2},u_{2,1},\ldots,u_{2,h-2},u_{1,1},\ldots,u_{1,h-1}$ are 
all the integer points contained in $\Pc_{h,d} \cap \ZZ^d$ except for the vertices. 
Note that $u_{3,0}=u_3, u_{3,h-1}=u_6, u_{2,0}=u_2, u_{2,h-1}=u_5, u_{1,0}=u_1$ and $u_{1,h}=u_4$. 

\smallskip

Let $A_{h,d} \in \ZZ^{d \times (2(d-2)+3h+1)}$ be an integer matrix of the form 
\begin{eqnarray*}
\big( v_2^*, v_2'^*, \cdots, v_{d-1}^*, v_{d-1}'^*, 
u_{3,0}^*, \ldots, u_{3,h-1}^*, 
u_{2,0}^*, \ldots, u_{2,h-1}^*, u_{1,0}^*, \ldots, u_{1,h}^* \big), 
\end{eqnarray*}
where $v^*=\eb_1+v$ for an integer point $v \in \ZZ^d$. 
This is nothing but a {\em configuration} arising from $\Fc_0$. 
Let $K[T]=K[t_1,t_2,\ldots,t_d]$ be the polynomial ring in 
$d$ variables over a field $K$. Then the toric ring of $A_{h,d}$ 
is the subalgebra $K[A_{h,d}]$ of $K[T]$ which is generated by the monomials 
\begin{align*}
&t_1t_2, t_1t_2t_d, \ldots, t_1t_{d-1}, t_1t_{d-1}t_d, \;\;\;\; 
t_1t_2 \cdots t_{d-1}, t_1t_2^2 \cdots t_{d-1}^2t_d, \ldots, t_1t_2^h \cdots t_{d-1}^ht_d^{h-1}, \\
&t_1t_d, t_1t_2 \cdots t_{d-1}t_d^2, \ldots, t_1t_2^{h-1} \cdots t_{d-1}^{h-1}t_d^h, \;\;\;\; 
t_1, t_1t_2 \cdots t_{d-1}t_d, \ldots, t_1t_2^h \cdots t_{d-1}^ht_d^h. 
\end{align*}
Let $K[X,Y,Z,W]=K[x_1,\ldots,x_{2d-4},y_1,\ldots,y_h,z_1,\ldots,z_h,w_0,w_1,\ldots,w_h]$ 
be the polynomial ring in $2d+3h-3$ variables over $K$ and define the surjective ring homomorphism 
$\pi : K[X,Y,Z,W] \rightarrow K[A_{h,d}]$ by setting 
\begin{align*}
&\pi(x_{2i-1})=t_1t_{i+1}, \;\; \pi(x_{2i})=t_1t_{i+1}t_d &
\quad&\text{for}\;\; i=1,\ldots,d-2, \\
&\pi(y_j)=t_1t_2^j \cdots t_{d-1}^jt_d^{j-1}, \;\; 
\pi(z_j)=t_1t_2^{j-1} \cdots t_{d-1}^{j-1}t_d^j & &\text{for}\;\; j=1,\ldots,h, \\
&\pi(w_k)=t_1t_2^k \cdots t_{d-1}^kt_d^k & &\text{for}\;\; k=0,\ldots,h. 
\end{align*}
The toric ideal $I$ is the kernel of the map $\pi$. 
Let $<$ be the lexicographic order on $K[X,Y,Z,W]$ induced by the ordering 
$$w_h < \cdots < w_0 < z_h < \cdots < z_1 < y_h < \cdots < y_1 < 
x_{2d-4} < \cdots < x_1.$$ 

\begin{Proposition}\label{gb}
A Gr\"obner basis of $I$ with respect to $<$ consists of the sets $G_1,\ldots,G_8$ 
of the binomials, where 
\begin{align*}
&G_1=\{x_{2i-1}x_{2j}-x_{2i}x_{2j-1} : 1 \leq i < j \leq d-2\}, \\
&G_2=\{y_iy_l - y_jy_k,z_iz_l - z_jz_k 
: 1 \leq i \leq j \leq k \leq l \leq h \;\text{{\em with}}\; i+l=j+k\}, \\
&G_3=\{w_iw_l - w_jw_k : 0 \leq i \leq j \leq k \leq l \leq h \;\text{{\em with}}\; i+l=j+k\}, \\
&G_4=\{x_{2i-1}z_j-x_{2i}w_{j-1},x_{2i-1}w_j-x_{2i}y_j : 1 \leq i \leq d-2, 1 \leq j \leq h \}, \\
&G_5=\{y_iz_j-w_{i-1}w_j : 1 \leq i,j \leq h\}, \\
&G_6=\{y_iw_j-y_{i+1}w_{j-1}, z_iw_j-z_{i+1}w_{j-1} : 1 \leq i \leq h-1, 1 \leq j \leq h\}, \\
&G_7=\{x_{2i-1}y_jw_0-x_{2i}y_1y_{j-1} : 1 \leq i \leq d-2, 2 \leq j \leq h\}, \\
&G_8=\left\{\prod_{q=1}^kx_{2q-1}\prod_{q=k+1}^{d-2}x_{2q} - z_1^{d-4-k}z_2w_0^{k+1} : 0 \leq k \leq d-4 \right\} \\
&\quad\quad\quad\quad\quad\quad\quad\quad\quad\quad\quad \bigcup 
\left\{x_{2d-4}\prod_{q=1}^{d-3}x_{2q-1} - w_0^{d-3}w_1, \prod_{q=1}^{d-2}x_{2q-1} - w_0^{d-3}y_1 \right\}. 
\end{align*}
\end{Proposition}
\begin{proof}
Let $\Gc=\bigcup_{i=1}^8G_i$. 
Let $\ini_<(G_i)$ denote the set of the initial monomials of all the binomials in $G_i$ with respect to $<$ 
and $\ini_<(\Gc)$ the ideal generated by all the monomials in $\bigcup_{i=1}^8\ini_<(G_i)$. 
Here the initial monomial of each binomial in $G_i$ is the first monomial. 
Since $\Gc \subset I$, we have $\ini_<(\Gc) \subset \ini_<(I)$. 
Our goal is to show $\ini_<(I) \subset \ini_<(\Gc)$. 

Fix an irreducible non-zero binomial $f=u-v \in I$ with $v < u$. 
Thus $u \in \ini_<(I)$. 
For monomials $m_1, m_2 \in K[X,Y,Z,W]$, let $m_1 \mid m_2$ (resp. $m_1 \nmid m_2$) 
denote that $m_2$ is divisible (resp. not divisible) by $m_1$. 
Suppose that $u \not\in \ini_<(\Gc)$. 

First, we assume that $x_i \nmid u$ for any $1 \leq i \leq 2d-4$. 
\begin{itemize}
\item Assume that $y_j \nmid u$ for any $1 \leq j \leq h$. 
Then, for any $i,j$, both $x_i \nmid v$ and $y_j \nmid v$ are satisfied. 
\begin{itemize}
\item When $z_j \nmid u$ for any $j$, 
the variables appearing in $u$ are only $w_k, 0 \leq k \leq h,$ 
and so is $v$. Since $u-v \in I$, it must be $u \in \ini_<(G_3)$, a contradiction. 
\item When $z_j \mid u$ for some $j$, since $u \not\in \ini_<(G_2)$, 
the variable among $z_j$ appearing in $u$ is either $z_j$ or $z_jz_{j+1}$. 
When the former case, i.e., when $u=z_j^{c_j}\prod w_k$, 
since $u \not\in \ini_<(G_6)$, we have $j=h$ or $\prod w_k=w_0^{d_1}$. 
If $j=h$, since $f$ is irreducible, only $w_k$ appears in $v$, 
which contradicts to $f \in I$. 
Similarly, if $\prod w_k=w_0^{d_1}$, then it contradicts $f \in I$. 
When the latter case, i.e., when $u=z_j^{c_j}z_{j+1}^{c_{j+1}} \prod w_k$, 
similarly, it contradicts $f \in I$. 
\end{itemize}
\item Assume that $y_j \mid u$ for some $j$. 
Then, since $u \not\in \ini_<(G_2) \cup \ini_<(G_5) \cup \ini_<(G_6)$, $u$ looks like 
$y_j^{b_j} w_0^{d_1}$, $y_j^{b_j}y_{j+1}^{b_{j+1}} w_0^{d_1}$ or $y_h^{b_h} \prod w_k$. 
In these cases, similar discussions to the previous case can be applied and 
lead a contradiction. 
\end{itemize}

Next, we assume that $x_i \mid u$ for some $i$. 
\begin{itemize}
\item When only one variable $x_i$ appears in $u$, 
$u$ looks like $x_i^{a_i} \prod y_j \prod z_j \prod w_k$. 
\begin{itemize}
\item When $i$ is even, the variables appearing in $v$ are chosen from $x_{i+1},\ldots,x_{2d-4}$, 
which obviously contradicts $f \in I$. 
\item When $i$ is odd, since $u \not\in \ini_<(G_4) \cup \ini_<(G_7)$, $u$ looks like either 
$x_i^{a_i} \prod y_j$ or $x_i^{a_i}y_1^{b_1}w_0^{d_1}$. 
When these cases, it contradicts $f \in I$. 
\end{itemize}
\item When at least $(d-1)$ distinct $x_i$'s appear in $u$, there is at least one $1 \leq q \leq d-2$ 
such that $x_{2q-1}x_{2q} \mid u$, which contradicts $f \in I$. 
\item When there are $(d-2)$ $x_i$'s in $u$ and there is no $q$ such that $x_{2q-1}x_{2q} \mid u$, 
one has $u \in \ini_<(G_8)$, a contradiction. 
When there are distinct $r$ $x_i$'s in $u$, where $2 \leq r \leq d-3$, 
and there is no $q$ such that $x_{2q-1}x_{2q} \mid u$, 
since $u \not\in \ini_<(G_1)$, $u$ looks like 
$x_{2i_1}^{a_{2i_1}} \cdots x_{2i_l}^{a_{2i_l}} x_{2i_{l+1}-1}^{a_{2i_{l+1}-1}} \cdots x_{2i_r-1}^{a_{2i_r-1}} 
\prod y_j \prod z_j \prod w_k$, where $1 \leq i_1 < \cdots <i_r \leq d-2$. 
This contradicts $f \in I$. 
\end{itemize}

Therefore, we conclude that $u$ belongs to $\ini_<(\Gc)$, as required. 
\end{proof}

\begin{Corollary}\label{cor}
The integral convex polytope $\Fc_0$ has a regular unimodular triangulation. 
In particular, $\Fc_0$ is normal. 
\end{Corollary}
\begin{proof}
By Proposition \ref{gb}, the toric ideal $I$ has a squarefree initial ideal. 
This is equivalent to what $\Fc_0$ has a regular unimodular triangulation. 
(Consult, e.g., \cite[Corollary 8.9]{Sturmfels}.) 
In general, an integral convex polytope having a unimodular triangulation is normal. 
\end{proof}

The remaining facets of $\Pc_{h,d}$ are also normal. 
\begin{Lemma}\label{unimod}
The facets $\Fc_1$ and $\Fc_{j,i}$, where $j=2,\ldots,9$ and $i=2,\ldots,d-1$, are all normal. 
\end{Lemma}
\begin{proof}
First, let us discuss the facets $\Fc_{4,i}, \Fc_{7,i}$ and $\Fc_{8,i}$. 
Fix $i=2$. Then the sets of vertices of $\Fc_{4,2}, \Fc_{7,2}$ and $\Fc_{8,2}$ 
are $\{u_1,u_7,u_9,v_3,\ldots,v_{d-1}\}$, $\{u_4,u_8,u_{10},v_3',\ldots,v_{d-1}'\}$ 
and $\{u_4,u_5,u_8,v_3',\ldots,v_{d-1}'\}$, respectively. 
Each of the matrices whose column vectors are the vertices of each facet 
can be transformed into the matrix $({\bf 0},\eb_1,\ldots,\eb_{d-1})$ by unimodular transformations. 
Thus each facet $\Fc_{4,2}, \Fc_{7,2}$ or $\Fc_{8,2}$ is 
unimodularly equivalent to a unit simplex of dimension $d-1$. 
Thus, in particular, $\Fc_{4,2}, \Fc_{7,2}$ and $\Fc_{8,2}$ are normal. 
Similarly, for any $i$, $\Fc_{4,i}, \Fc_{7,i}$ and $\Fc_{8,i}$ are normal. 

Next, let us investigate $\Fc_1,\Fc_{2,i},\Fc_{6,i}$ and $\Fc_{9,i}$. 
For $\Fc_{2,2}$, the set of its vertices is 
$$\{u_1,u_2,u_7,u_8,v_3,\ldots,v_{d-1},v_3',\ldots,v_{d-1}'\}.$$ 
By unimodular transformations, 
the matrix whose column vector is the above vertices 
can be transformed into the matrix 
$$({\bf 0}, -\eb_d, \eb_1, \eb_1-\eb_d, \eb_3,\ldots,\eb_{d-1},\eb_3-\eb_d,\ldots,\eb_{d-1}-\eb_d).$$ 
This is totally unimodular (\cite[Chapter 19]{Sch}). 
Thus, $\Fc_{2,2}$ has a unimodular triangulation. In particular, this is normal. 
Similarly, for any $i$, $\Fc_{2,i}$ is normal and so are $\Fc_1$, $\Fc_{6,i}$ and $\Fc_{9,i}$. 

Finally, let us consider the facets $\Fc_{3,i}$ and $\Fc_{5,i}$. 
Then one can see that each of them is unimodularly equivalent to 
the simplex whose vertex set is $\{{\bf 0},\eb_1,\ldots,\eb_{d-2},(h-1)\eb_{d-1}\}$. 
This is also normal, as desired. 
\end{proof}

\section{Holes of $\Pc_{h,d}$}\label{Hilb}


Let 
$$(u_j',2)=\frac{1}{2}((u_{1,j-1},1)+(u_{1,j},1)+(u_8,1)+(u_9,1)) \;\text{ for }\; j=1,\ldots,h. $$ 
Then $u_j'=\eb_1+j(\eb_2+\cdots+\eb_{d-1})+(j+2)\eb_d$ and 
each $(u_j',2)$ is contained in $\RR_{\geq 0}\Ac_{\Pc_{h,d}} \cap \ZZ^{d+1}$. 
On the other hand, since none of the points $(u_j',2)-(u_i,1)$, where $i=7,\ldots,10$, 
are contained in $\Ac_{\Pc_{h,d}}$, it must be $(u_j',2) \not\in \ZZ_{\geq 0}\Ac_{\Pc_{h,d}}$. 
Hence, the above $h$ integer points are holes of $\Pc_{h,d}$. 
In the rest of this section, we show that there is no more holes in 
\begin{eqnarray}\label{2part}
\{x \in \RR_{\geq 0} \Ac_{\Pc_{h,d}} \cap \ZZ^{d+1} : \deg(x)=2\}.
\end{eqnarray}

For nonnegative integers $n$ and $k$, let 
$$\Pc_{h,d}(n,k)=n \Pc_{h,d} \cap \{(x_1,\ldots,x_d) \in \RR^d : x_1=k\}.$$ 
For example, $\Pc_{h,d}(1,0)=\Fc_0$ and $\Pc_{h,d}(1,1)=\con(\{u_7,u_8,u_9,u_{10}\})$. 
Let $\Pc_0=\Pc_{h,d}(1,0)$ and $\Pc_1=\Pc_{h,d}(1,1)$. 

For a hyperplane $\Hc$ defined by the equality $a_1x_1+\cdots+a_dx_d = b$ and a positive integer $m$, 
we write $m\Hc^{(+)}$ (resp. $m\Hc^{(-)}$) for the closed half space 
defined by the inequality $a_1x_1+\cdots+a_dx_d \leq mb$ (resp. $a_1x_1+\cdots+a_dx_d \geq mb$). 


\begin{Lemma}\label{deg2}
Let $\Qc=\Pc_{h,d}(2,1)$. Then one has 
$$\Qc \cap \ZZ^d = \{{\bf a_0}+{\bf a_1} \in \ZZ^d : {\bf a_i} \in \Pc_i \cap \ZZ^d, i=0,1 \} 
\cup \{u_1',\ldots,u_h'\}.$$
\end{Lemma}
\begin{proof}
Clearly, $\Qc \cap \ZZ^d 
\supset \{{\bf a_0}+{\bf a_1} \in \ZZ^d : {\bf a_i} \in \Pc_i \cap \ZZ^d, i=0,1 \} \cup \{u_1',\ldots,u_h'\}.$
Thus we may show the other inclusion. We remark that from \eqref{facets}, one has 
$$\Qc \subset 2 \Hc_0^{(-)} \cap 2 \Hc_1^{(-)} \cap 
\left(\bigcap_{\substack{2 \leq j \leq 9, \\ 2 \leq i \leq d-1}} 2 \Hc_{j,i}^{(+)} \right).$$

Let $x=(1,x_2,\ldots,x_d) \in \Qc \cap \ZZ^d$. \\
{\em The first step.} 
Assume that $x_2=x_3=\cdots=x_{d-1}$. 
Since $x \in 2 \Hc_{2,2}^{(+)} \cap 2 \Hc_{9,2}^{(+)}$, 
one has $0 \leq x_2 \leq h+1$. 
On the other hand, since $$x \in \bigcap_{\substack{j=3,4,5,7,8, \\ 2 \leq i \leq d-1}} 2\Hc_{j,i}^{(+)},$$ 
we have $$\max\{x_2-2, -4(x_2-1)\} \leq x_d \leq \min\{x_2+6,5h+5-4x_2, h+5\}.$$ 
One can verify that all of these are contained in 
$\{{\bf a_0}+{\bf a_1} \in \ZZ^d : {\bf a_i} \in \Pc_i \cap \ZZ^d, i=0,1 \} 
\cup \{u_1',\ldots,u_h'\}$. \\
{\em The second step.} 
Assume that $x$ does not satisfy $x_2=x_3=\cdots=x_{d-1}$. 
Let $a_1,\ldots,a_m$ be distinct $m$ integers such that 
$\{a_1,\ldots,a_m\}=\{x_2,\ldots,x_{d-1}\}$, where $a_1 > a_2 > \cdots > a_m$. 
Then $m \geq 2$. By $x \in \bigcap_{i=2}^{d-1} 2 \Hc_{2,i}^{(+)}$, 
we have $a_m \geq 0$. 
Let $p_\ell$ be the number of $a_\ell$'s among $x_2,\ldots,x_{d-1}$. 
Thus, $p_\ell>0$ and $p_1+\cdots+p_m=d-2$. For $\ell=1,\ldots,m-1$, let $b_\ell=a_\ell-a_m$. 
Then $b_\ell \geq m-\ell$. 
From $x \in \bigcap_{i=2}^{d-1} 2 \Hc_{6,i}^{(+)}$, we have 
\begin{eqnarray*}
&&-(d-3)a_m+(p_1a_1+\cdots+p_{m-1}a_{m-1}+(p_m-1)a_m) \\
&&\quad =-(d-3)a_m+(p_1(a_m+b_1)+\cdots+p_{m-1}(a_m+b_{m-1})+(p_m-1)a_m) \\
&&\quad =-(d-2)a_m+(p_1+\cdots+p_m)a_m+p_1b_1+\cdots+p_{m-1}b_{m-1} \\
&&\quad =p_1b_1+\cdots+p_{m-1}b_{m-1} \leq 1. 
\end{eqnarray*}
Hence, we obtain $m=2$ and $p_1=b_1=1$. Let, say, $x=(1,a_m+1,a_m,\ldots,a_m,x_d)$. 
Moreover, from $x \in \bigcap_{i=2}^{d-1} 2 \Hc_{9,i}^{(+)}$, we have 
$$-((d-3)h-1)a_m+h(a_m+1+(d-4)a_m)=h+a_m \leq h+1, $$ 
which implies that $a_m=0$ or 1. 
\begin{itemize}
\item When $x=(1,1,0,\ldots,0,x_d)$, since $x \in 2 \Hc_{4,3}^{(+)} \cap 2 \Hc_{5,3}^{(+)}$, 
we have $4 \leq x_d \leq 6$. 
\item When $x=(1,2,1,\ldots,1,x_d)$, since $x \in 2 \Hc_1^{(-)} \cap 2 \Hc_{7,3}^{(+)}$, 
we have $0 \leq x_d \leq 2$. 
\end{itemize}
All of these are contained in 
$\{{\bf a_0}+{\bf a_1} \in \ZZ^d : {\bf a_i} \in \Pc_i \cap \ZZ^d, i=0,1 \}$. 
Similarly, the integer points $x=(1,a_m,\ldots,a_m,x_d) + \eb_j$, 
where $a_m \in \{0,1\}$ and $j=3,\ldots,d-1$, 
are also contained there, as required. 
\end{proof}

Now, Corollary \ref{cor} says that $\Pc_0$ is normal. 
Moreover, since $\Pc_1$ is of dimension 2, this is also normal (\cite[Corollary 2.2.13]{CLS}). 
Thus, there is no hole in 
$$\{(x_1,\ldots,x_d,x_{d+1}) \in \RR_{\geq 0}\Ac_{\Pc_{h,d}} \cap \ZZ^{d+1} : x_1 \in \{0,2\}, x_{d+1}=2\}.$$

Therefore, there exist exactly $h$ holes contained in \eqref{2part}.

\section{The 3-normality of $\Pc_{h,d}$}\label{nohole}

In this section, we claim that there is no other hole except for $(u_j',2)$, $j=1,\ldots,h$. 
In other words, we prove that $\Pc_{h,d}$ is 3-normal.

Similar computations to Lemma \ref{deg2} enable us to show the following 
\begin{Lemma}\label{first}
One has 
\begin{itemize}
\item[(a)] $\Pc_{h,d}(3,1) \cap \ZZ^d = \{{\bf a_0}+{\bf a_0}'+{\bf a_1} \in \ZZ^d 
: {\bf a_0},{\bf a_0}' \in \Pc_0 \cap \ZZ^d , {\bf a_1} \in \Pc_1 \cap \ZZ^d\}$; 
\item[(b)] $\Pc_{h,d}(3,2) \cap \ZZ^d = \{{\bf a_0}+{\bf a_1}+{\bf a_1}' \in \ZZ^d 
: {\bf a_0} \in \Pc_0 \cap \ZZ^d , {\bf a_1},{\bf a_1}' \in \Pc_1 \cap \ZZ^d\}$; 
\item[(c)] $\Pc_{h,d}(4,1) \cap \ZZ^d = \{{\bf a_0}+{\bf a_0}'+{\bf a_0}''+{\bf a_1} \in \ZZ^d : 
{\bf a_0},{\bf a_0}',{\bf a_0}'' \in \Pc_0 \cap \ZZ^d , {\bf a_1} \in \Pc_1 \cap \ZZ^d\}$. 
\end{itemize}
\end{Lemma}

Finally, we prove 
\begin{Lemma}
Let $n \geq 3$ and $0 \leq k \leq n$. For each $\alpha \in \Pc_{h,d}(n,k)$, we have 
\begin{eqnarray}\label{express}
\alpha={\bf a_0}^{(1)}+\cdots+{\bf a_0}^{(n-k)}+{\bf a_1}^{(1)}+\cdots+{\bf a_1}^{(k)},
\end{eqnarray}
where ${\bf a_0}^{(s)} \in \Pc_0 \cap \ZZ^d$ for $s=1,\ldots,n-k$ 
and ${\bf a_1}^{(t)} \in \Pc_1 \cap \ZZ^d$ for $t=1,\ldots,k$. 
That is to say, $\Pc_{h,d}$ is 3-normal. 
\end{Lemma}
\begin{proof}
Fix $\alpha=(\alpha_1,\ldots,\alpha_d) \in \Pc_{h,d}(n,k)$, where $\alpha_1=k$. 
Since $\Pc_0$ and $\Pc_1$ are normal, we may assume that $1 \leq k \leq n-1$. 
Moreover, thanks to Lemma \ref{first}, we may also assume that 
$n \geq 4, k \geq 2$ or $n \geq 5, k=1$. 
In addition, by Lemma \ref{unimod}, we may also assume that 
$$\alpha \not\in n\Fc_0 \cup n\Fc_1 \cup 
\left(\bigcup_{\substack{2 \leq j \leq 9, \\ 2 \leq i \leq d-1}} n\Fc_{j,i} \right).$$ 
We will proceed our discussions by induction on $n$. \\
{\em The first step.} 
Suppose that $\alpha$ satisfies the following $(d-1)$ inequalities: 
\begin{align}\label{condition}
\alpha_d \geq 5 \;\;\text{and}\;\; 
 -(d-4)\alpha_i+\sum_{j \not= i, 2 \leq j \leq d-1}\alpha_j-\alpha_d \leq n-6 
\;\;\text{for}\;\;i=2,\ldots,d-1. 
\end{align}
Let $\beta=\alpha-u_8=(\alpha_1-1,\alpha_2,\ldots,\alpha_{d-1},\alpha_d-5)$. 
Then we have $\beta \in \Pc_{h,d}(n-1,k-1) \cap \ZZ^d$. 
In fact, one can easily see that for $i=2,\ldots,d-1$, we have 
\begin{itemize}
\item $\alpha_1-1=k-1 \geq 0$; 
\item $\alpha_d-5 \geq 0$ by \eqref{condition}; 
\item $\alpha_i \geq 0$; 
\item $-(d-4)\alpha_i+\sum_{j\not=i, 2\leq j \leq d-1}\alpha_j-(\alpha_d-5) 
\leq n-6 +5=n-1$ by \eqref{condition}; 
\item $4(\alpha_1-1)-4\alpha_i-(\alpha-5) \leq -1-4+5=0$ since $\alpha \not\in n\Fc_{4,i}$; 
\item $-4(\alpha_1-1)-\alpha_i+\alpha_d-5 \leq n-1$; 
\item $\alpha_1-1-(d-3)\alpha_i+\sum_{j\not=i, 2\leq j \leq d-1}\alpha_j \leq n-1$; 
\item $(5h-5)(\alpha_1-1)-((d-3)(5h-1)-4)\alpha_i+(5h-1)\sum_{j\not=i, 2 \leq j \leq d-1}\alpha_j + \alpha_d-5 
\leq 5hn-(5h-5)-5 = 5h(n-1)$; 
\item $(h-5)(\alpha_1-1)-(d-3)(h-1)\alpha_i+(h-1)\sum_{j\not=i, 2 \leq j \leq d-1}\alpha_j + \alpha_d-5 
\leq hn-(h-5)-5=h(n-1)$; 
\item $(h-1)(\alpha_1-1)-((d-3)h-1)\alpha_i+h\sum_{j\not=i, 2 \leq j \leq d-1}\alpha_j \leq hn-1-(h-1) = h(n-1)$ 
since $\alpha \not\in n\Fc_{9,i}$. 
\end{itemize}
The above estimations imply that $\beta \in (n-1)\Pc_{h,d} \cap \ZZ^d$ because of \eqref{facets}. 
By the hypothesis of induction, we can obtain the required expression on $\alpha$ like \eqref{express}. \\
{\em The second step.} 
Suppose that $\alpha$ satisfies either 
\begin{eqnarray}\label{condi1}
\alpha_d \leq 4 \;\text{ or }\;
-(d-4)\alpha_i+\sum_{\substack{j \not= i, \\ 2 \leq j \leq d-1}}\alpha_j-\alpha_d \geq n-5 
\;\text{for}\;i=2,\ldots,d-1. 
\end{eqnarray}
Then we obtain the new inequalities 
\begin{eqnarray}\label{newcondition}
&& -4\alpha_1-\alpha_i+\alpha_d \leq n -6, \\
\nonumber 
&& (h-5)\alpha_1-(d-3)(h-1) \alpha_i+(h-1)\sum_{j\not=i,2\leq j \leq d-1}\alpha_j+\alpha_d 
\leq hn-5
\end{eqnarray}
for $i=2,\ldots,d-1$ as follows. 
\begin{itemize}
\item[(i)] First, suppose that $\alpha$ satisfies the left-hand condition of \eqref{condi1}. 
Since $\alpha \not\in n \Fc_{2,i}$, 
one has $\alpha_i \geq 1 \geq 10-4k-n$ 
from our assumption $n \geq 4, k \geq 2$ or $n \geq 5, k=1$. Thus we obtain 
\begin{align*}
-4\alpha_1-\alpha_i+\alpha_d \leq -4k-10+4k+n+4=n-6. 
\end{align*}
Moreover, since $\alpha \not\in n\Fc_{6,i}$, one has 
$\alpha_1-(d-3)\alpha_i + \sum_{j\not=i, 2 \leq j \leq d-1}\alpha_j \leq n-k-1$. 
Hence $(h-1)(\alpha_1-(d-3)\alpha_i + \sum_{j\not=i, 2 \leq j \leq d-1}\alpha_j) \leq (h-1)(n-1)+4k+h+n-10$. 
Remark that $h \geq 1$. Thus we also obtain 
\begin{align*}
(h-5)\alpha_1&-(d-3)(h-1)\alpha_i+(h-1)\sum \alpha_j +\alpha_d \\
&\quad\quad\quad\leq -4\alpha_1+(h-1)(n-1)+4k+h+n-10+4= hn-5. 
\end{align*}
\item[(ii)] Second, suppose that $\alpha$ satisfies the right-hand condition of \eqref{condi1}. 
Since $\alpha \not\in n\Fc_{6,i}$, one has 
$-4\alpha_1 - (d-3)\alpha_i + \sum \alpha_j \leq n-5k-1 \leq n-5k-1+n+5k-10=2n-11.$ Thus we obtain 
\begin{align*}
-4\alpha_1-\alpha_i+\alpha_d &= -4\alpha_1-(d-3)\alpha_i+\sum \alpha_j-(-(d-4)\alpha_i+\sum \alpha_j-\alpha_d) \\
&\leq 2n-11-n+5=n-6.
\end{align*}
Moreover, since $\alpha \not\in n\Fc_{9,i}$, one has 
$(h-1)\alpha_1-((d-3)h-1)\alpha_i+h\sum \alpha_j \leq hn-1 \leq hn-1+4k+n-9=hn+4k+n-10$. 
Thus we obtain 
\begin{align*}
&(h-5)\alpha_1-(d-3)(h-1)\alpha_i+(h-1)\sum \alpha_j +\alpha_d = \\
&\quad\quad -4\alpha_1+(h-1)\alpha_1-((d-3)h-1)\alpha_i+h\sum \alpha_j -(-(d-4)\alpha_i+\sum \alpha_j-\alpha_d) \\
&\quad\quad \leq -4k+hn+4k+n-10-n+5=hn-5. 
\end{align*}
\end{itemize}
Let $\beta'=\alpha-u_9$. If we assume that $\alpha$ satisfies \eqref{condi1} 
then similar to the first step, 
we can verify that $\beta' \in (n-1)\Pc_{h,d} \cap \ZZ^d$. 
Here we use \eqref{newcondition} and the normality of 
some facets of $\Pc_{h,d}$ in the same way as the first step. \\
{\em The third step.} 
Suppose that $\alpha$ satisfies neither \eqref{condition} nor \eqref{condi1}. 
When this is the case, one has $d \geq 4$ and 
there exist $\ell$ and $\ell'$ with $2 \leq \ell \not= \ell' \leq d-1$ 
such that the inequalities 
\begin{eqnarray*}
\quad\quad -(d-4)\alpha_{\ell}+\sum\alpha_j-\alpha_d \leq n-6 \;\text{ and }\; 
-(d-4)\alpha_{\ell'}+\sum\alpha_j-\alpha_d \geq n-5 
\end{eqnarray*}
are satisfied. 
It then follows that $(d-3)(\alpha_{\ell}-\alpha_{\ell'}) \geq 1$, 
i.e., $\alpha_{\ell}-\alpha_{\ell'} \geq 1$. Let $\beta''=\alpha-v_{\ell}$. 
Then, similarly, we can verify that $\beta'' \in (n-1)\Pc_{h,d} \cap \ZZ^d$ 
by using $\alpha_{\ell}-\alpha_{\ell'} \geq 1$ and the normality 
of some facets of $\Pc_{h,d}$, as desired. 
\end{proof}

\end{document}